\documentclass[a4paper, 11pt]{article}

\usepackage{amssymb, enumerate}
\usepackage{amsfonts}
\usepackage{amsmath}
\usepackage{color}
\usepackage{hyperref}
\usepackage{graphicx}
\usepackage{tikz}

\usepackage[left=2.5cm,right=2.5cm,bottom=3cm,top=2.5cm,headsep=0.5cm]{geometry}

\newcommand*{\leftmapsto}{\mathbin{\reflectbox{$\mapsto$}}}

\newcommand{\R}{\mathbb{R}}

\newcommand{\Z}{\mathbb{Z}}

\newcommand{\wt}{\widetilde}

\newcommand{\bg}{\mathbf{g}}
\newcommand{\bfB}{\mathbf{B}}
\newcommand{\bfC}{\mathbf{C}}

\newtheorem{proposition}{Proposition}
\newtheorem{theorem}[proposition]{Theorem}

\newtheorem{remark}[proposition]{Remark}

\newtheorem{lemma}[proposition]{Lemma}
\newtheorem{example}{Example}

\newtheorem{assumption}{Assumption}

\newenvironment{proof}[1]{
  \trivlist \item[\hskip \labelsep{\it #1}]}{\hfill\mbox{$\square$}
  \endtrivlist}

\title{Rational certificates of non-negativity on semialgebraic subsets of cylinders
}
\author{Gabriela Jeronimo$^{1,2,3,}$\footnote{Partially supported by Universidad de Buenos Aires (UBACYT 20020190100116BA),  CONICET (PIP 2021-2023 GI 11220200101015CO) and Agencia Nacional de Promoci\'on Cient\'ifica y Tecnol\'ogica (PICT 2018--02315), Argentina.}
\and Daniel Perrucci$^{1,2,*}$}
\date{}

\begin{document}

\maketitle

{\small
\noindent $^1$ Universidad de Buenos Aires, Facultad de Ciencias Exactas y Naturales, Departamento de Matem\'atica. Buenos Aires, Argentina. \\
\noindent $^2$ CONICET -- Universidad de Buenos Aires, Instituto de Investigaciones Matem\'aticas ``Luis A. Santal\'o'' (IMAS), Buenos Aires, Argentina.\\
$^3$ Universidad de Buenos Aires, Ciclo B\'asico Com\'un, Departamento de Ciencias Exactas. Buenos Aires, Argentina.
}

\begin{abstract}
Let $g_1,\dots, g_s \in \R[X_1,\dots, X_n,Y]$ and
$S = \{(\bar{x},y)\in \R^{n+1} \mid g_1(\bar{x},y) \ge 0, \dots, g_s(\bar{x}, y) \ge 0\}$
be a non-empty, possibly unbounded, subset of a cylinder in $\R^{n+1}$.
Let $f \in \R[X_1, \dots, X_n, Y]$ be a polynomial which is positive on $S$.
We prove that, under certain additional assumptions,
for any non-constant polynomial $q \in \R[Y]$ which is positive on $\R$, there is
a
certificate of the non-negativity of $f$ on $S$ given by a rational function having
as numerator a polynomial in the quadratic module generated by $g_1, \dots, g_s$ and
as denominator a power of $q$.
\end{abstract}



\section{Introduction}

Certificates of positivity and non-negativity by means of sums of squares is a topic whose roots go back to Hilbert's 17-th problem and its celebrated solution by Artin (\cite{Artin}).
Another milestone in the development of this theory is
the Positivstellensatz by Krivine (\cite{Kri}) and Stengle (\cite{Ste}) which
provides rational certificates for a multivariate polynomial $f$ that is positive
or non-negative on a
basic closed semialgebraic set $S \subset \R^n$.

More recently,
the famous works by Schm\"udgen (\cite{Schm}) and Putinar (\cite{Put}) led to a renewed interest in
these certificates.
Schm\"udgen's Positivstellensatz
ensures the existence of a polynomial certificate of non-negativity for a polynomial $f$
which is positive on a compact set $S$.
Under a stronger assumption which implies compactness of $S$,
Putinar's Positivstellensatz establishes the existence of a simpler polynomial
certificate.
Several subsequent works extended the previous theorems in different directions including
non-compact situations.
We refer the reader to the survey by Scheiderer (\cite{ScheidererSurvey}) and to the
books by Marshall (\cite{MarBook}) and by Powers (\cite{PowersBook}) for a comprehensive treatment
of the subject; see also \cite{SchmBook} for a 
specific reference on the moment problem and its connections with certificates of non-negativity.

In this paper, we address the problem of the existence of certificates of non-negativity in
the following setting.
Let $g_1,\dots, g_s \in \R[\bar{X},Y]= \R[X_1,\dots, X_n,Y]$ and
$$
S = \{(\bar{x},y)\in \R^{n+1} \mid g_1(\bar{x},y) \ge 0, \dots, g_s(\bar{x}, y) \ge 0\}
$$
be a non-empty, possibly unbounded, subset of a cylinder in $\R^{n+1}$.
Let $f \in \R[\bar X, Y]$ be a polynomial which is positive on $S$.
Under certain additional assumptions, we prove that,
for any non-constant polynomial $q \in \R[Y]$ which is positive on $\R$, there is
a
certificate of the non-negativity of $f$ on $S$ given by a rational function having
as numerator a polynomial in the quadratic module generated by $g_1, \dots, g_s$ and
as denominator a power of $q$
(see Theorem \ref{thm:main}).

Variants of this problem have been considered previously.
In \cite{Pow} and in \cite{EscPer},
Schm\"udgen's Positivstellensatz and
Putinar's Positivstellensatz are extended to cylinders with compact cross-section, under some additional
assumptions on the polynomial $f$.
In \cite{KuhlMar} and \cite{KuhlMarSchw2005}, among other results concerning non-negativity of polynomials on non-compact sets, the authors analyze the more general case of $S$ being a subset of a cylinder and they prove
the existence
of a polynomial certificate for a small suitable perturbation of $f$
provided that the polynomials defining $S$
satisfy certain assumptions.

On the other hand, the existence of rational certificates having as a denominator a power of a fixed
particular polynomial has been studied before in different frameworks.
In \cite{Reznick95}, it is proved that a polynomial $f$ which is positive on $\R^n$ is a sum of squares of rational functions having as denominators powers of
$1 + \sum X_j^2$.
Then in \cite{PutVasA} and \cite{PutVasB} this result is generalized to
basic closed semialgebraic sets, under additional assumptions to control the behavior of the polynomial $f$
at infinity.

Going back to Schm\"udgen's and Putinar's Positivstellensatz, in
\cite{Schw} and \cite{NieSchw2007} the authors develop a constructive approach in order to obtain bounds for the degrees
of every term involved in these certificates.
In this work, the main idea to prove Theorem \ref{thm:main} is, as in \cite{Pow} and \cite{EscPer}, to produce for each $y \in \R$ a certificate on the slice of $S$ cut by the equation $Y = y$, in a parametric way such that
all these certificates can be glued together in a single one.
The procedure we follow on each slice is indeed an adaptation of the
one in \cite{Schw} and \cite{NieSchw2007} using \cite{PowersReznick2001}.

This slicing method is somehow complementary to the one in \cite{KuhlMarSchw2005} where, 
instead, in order to deal with subsets of cylinders the fibres with respect to the projection on the variables $\bar X$ are considered to 
reduce the problem to the univariate case, and finally a clever glueing process is designed.
Furthermore, this approach is related to the fibre theorem proved by Schm\"udgen in \cite{Schm2003} (see also its generalization in \cite{Schm2017}), which reduces the moment problem for a closed (possibly unbounded)
basic semialgebraic set to the moment problem for its fibres with respect to a polynomial map with bounded
image.

The rest of the paper is organized in two sections. In Section \ref{sec:assump} we
introduce our assumptions and notation and state the main result, and then in Section
\ref{sec:proof} we prove it.

\section{Assumptions and main result}\label{sec:assump}

We introduce the notation we will use throughout the paper.

Let $g_1,\dots, g_s \in \R[\bar{X},Y]$.
We consider the quadratic module generated by $\bg:=(g_1,\dots, g_s)$:
\[M(\bg) = \left\{ \sigma_0 + \sigma_1 g_1 + \cdots +\sigma_s g_s \mid \sigma_0,\sigma_1, \dots, \sigma_s \in \sum \R[\bar{X}, Y]^2\right\},\]
that is, the smallest quadratic module in $\R[\bar{X},Y]$ that contains $g_1,\dots, g_s$.
As in \cite[Section 5]{KuhlMarSchw2005} (also \cite{EscPer}), we make an assumption on $M(\bg)$
which is weaker than Archimedianity but captures a similar idea on the variables $\bar X$.

\begin{assumption}\label{assump:arq}
There exists $N\in \R_{>0}$ such that
\[ N - \sum_{1\le j\le n} X_j^2 \in M(\bg).\]
\end{assumption}
Note that under this assumption, the set  $S$ is included in the cylinder with compact cross section $ \mathbf{B} \times \R$,
where
$$
\mathbf{B}:=\{ \bar{x} \in \R^n\mid \sum_{1\le j \le n} x_j^2 \le N\}.
$$

For $i=1,\dots, s$, if
\[g_i(\bar{X},Y)= \sum_{0\le k \le m_i} g_{ik}(\bar{X}) Y^k\in \R[\bar{X},Y] \]
with $g_{im_i}(\bar{X}) \ne 0$, we write
\[\wt{g}_i(\bar{X},Y,Z):= Z^{m_i} g_i(\bar{X},Y/Z)= \sum_{0\le k \le m_i} g_{ik}(\bar{X}) Y^k Z^{m_i-k}\in \R[\bar{X},Y,Z] \]
for the homogenization of $g_i$ with respect to the variable $Y$.
We make the following further assumptions on the polynomials $g_1, \dots, g_s$ and the set $S$ they describe.

\begin{assumption}\label{assump:g_i} {\ }
\begin{enumerate}
\item $S \ne \emptyset$.
\item For $i=1,\dots, s$, $m_i=\deg_Y(g_i)$ is even.
\item $S_\infty:=\{\bar{x} \in \R^n \mid g_{1 m_1}(\bar{x}) \ge 0, \dots, g_{s m_s}(\bar{x}) \ge 0 \} \subset \mathbf{B}$.
\end{enumerate}
\end{assumption}
Indeed, once Assumption $1$ is made, the third condition in Assumption \ref{assump:g_i} could be replaced by the condition that $S_{\infty}$
is bounded, since it is always possible to increase $N$ if necessary. Nevertheless, for simplicity we
assume that $N$ is big enough.
On the other hand, the following example shows that the third condition in Assumption \ref{assump:g_i} does not follow from
Assumption \ref{assump:arq} and the first two conditions in Assumption \ref{assump:g_i}.

\begin{example}
For $n = 2$, consider
$g_1 = X_1Y^2 + (1 - X_1^2 - X_2^2)$ and $g_2 = -X_1Y^2 + 1$.
Then
$$
2 - X_1^2 - X_2^2 = g_1 + g_1 \in M(\bg),
$$
so Assumption \ref{assump:arq} is satisfied. In addition $S \ne \emptyset$ (moreover, it is not bounded) since $(0, 0, y) \in S$ for every
$y \in \R$. However
$$
S_{\infty} = \{(x_1, x_2) \in \R^2 \ | \ x_1 = 0 \}
$$
is not a bounded set.
\end{example}

Let $q \in \R[Y]$ be a non-constant polynomial which is positive on $\R$ (and therefore, a sum of squares
in $\R[Y]$),
$$
q(Y) = \sum_{0 \le k \le m_0} q_kY^k
$$
with $m_0 > 0$ and $q_{m_0} \ne 0$. The assumption of $q$ being positive on $\R$ implies
$m_0$ is even and $q_{m_0} > 0$. We write
$$
\widetilde q(Y, Z) := Z^{m_0}q(Y/Z) = \sum_{0 \le k \le m_0} q_kY^kZ^{m_0-k}.
$$
Note that $\widetilde q(Y, Z)$ is a sum of squares in $\R[Y, Z]$, $\widetilde q$ is non-negative in $\R^2$ and
it only vanishes at the origin.

Let
$$
\bfC := \{(y,z) \in \R^2 \ | \ \widetilde q(y, z) = 1, z \ge 0\}
$$
and
\[\wt{S} := \{ (\bar{x},y,z) \in \R^{n+2} \mid \ \wt{g}_1(\bar{x},y,z) \ge 0, \dots, \wt{g}_s(\bar{x},y,z)\ge 0, \,
(y, z) \in \bfC \}
.\]

For $\theta \in [0, \pi]$ and $\rho \in \R$, we have that
$$
\widetilde q (\rho \cos(\theta), \rho \sin(\theta)) = \rho^{m_0}
\widetilde q( \cos(\theta), \sin(\theta)).
$$
Therefore, for any such $\theta$, there exists a unique $\rho(\theta) \in [0, +\infty)$ such that
$(\rho(\theta) \cos(\theta), \rho(\theta) \sin(\theta)) \in \bfC$, which is
$$
\rho(\theta) =   \widetilde q( \cos(\theta), \sin(\theta))^{-1/m_0}
$$
and satisfies $\rho(\theta) > 0$.
Since the function $\rho: [0, \pi] \to \R$ is continuous,
$$
\bfC = \{ (\rho(\theta)\cos(\theta), \rho(\theta)\sin(\theta) ) \ | \
\theta \in [0, \pi] \}
$$
is a compact set.
Moreover,
the set $\widetilde S \cap \{z \ne 0\}$ is in bijection with $S$.
This bijection is given by

$$
\begin{array}{ccc}
\widetilde S \cap \{z \ne 0\} & & S \\[2mm]
(\bar x, y, z) & \mapsto & (\bar x, y/z) \\[2mm]
\left( \bar x, \rho({\rm arccot}(y))\frac{y}{\sqrt{y^2  + 1}}, \rho({\rm arccot}(y))\frac{1}{\sqrt{y^2  + 1}}    \right)& \leftmapsto & (\bar x, y)
\end{array}
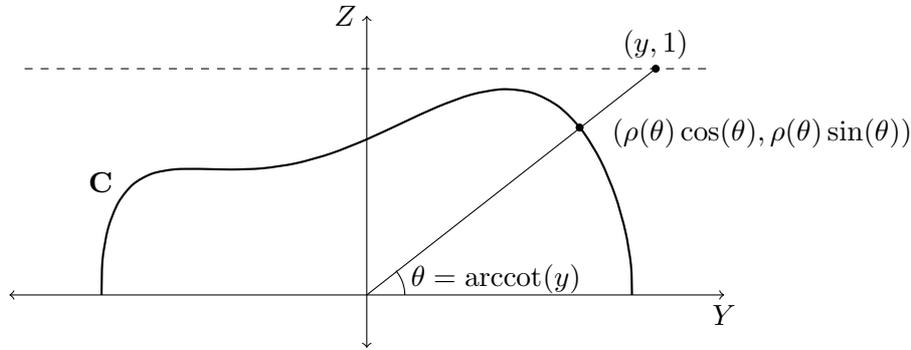
$$

\medskip

\begin{center}
\begin{tikzpicture}
      \draw[<->] (-4.7,0) -- (4.7,0) node[below] {$Y$};
      \draw[<->] (0,-0.7) -- (0,3.7) node[left] {$Z$};
      \draw[line width=0.8pt, domain= 0:180,smooth,variable=\x,] plot ( { 1.8*(2 - ((\x-90)/360)^2 )*cos(\x)  },
      { 1.8*(1.5 + sin(2.5*\x)/2 )*sin(\x) }  );
      \node at (-3.5,1.5) {$\bfC$};
      \draw[-] (0,0) -- (3.8,3) node[above]{$(y,1)$};
      \draw[line width=0.4pt, domain= 0:40,smooth,variable=\x,] plot ( {0.5*cos(\x)}, { 0.5*sin(\x) }  );
      \node at (1.7,0.22) {$\theta = {\rm arccot}(y)$};
      \draw[-, dashed] (-4.5,3) -- (4.5,3) ;
      \node at (5.2,2.15) {$(\rho(\theta)\cos(\theta), \rho(\theta)\sin(\theta))$};
      \fill (3.8, 3) circle[radius=1.5pt];
      \fill (2.8, 2.22) circle[radius=1.5pt];
      \end{tikzpicture}
\end{center}

\medskip

This implies
$$
\wt{S} \cap \{ z\ne0\}
\subset \bfB \times \bfC.
$$
On the other hand,
Assumption \ref{assump:g_i} implies
$$
\wt{S} \cap \{ z=0\} =
(S_\infty \times \{(-\rho(\pi),0)\})
\cup
(S_\infty \times \{(\rho(0),0)\})
\subset \bfB \times \bfC.
$$
We conclude that $\wt{S} \subset \bfB \times \bfC$
and therefore
$\wt{S}$ is compact.

Similarly,  for a polynomial
\[f(\bar{X},Y)= \sum_{0\le k \le m} f_{k}(\bar{X}) Y^k\in \R[\bar{X},Y] \]
with $f_m(\bar{X}) \ne 0$, we write
\[\wt{f}(\bar{X},Y,Z):= Z^{m} f(\bar{X},Y/Z)= \sum_{0\le k \le m} f_{k}(\bar{X}) Y^k Z^{m-k}\in \R[\bar{X},Y,Z] \]
for its homogenization with respect to the variable $Y$.

It is easy to see that $f$ is positive on $S$ if and only if $\wt{f}$ is positive on $\wt{S} \cap \{ z\ne 0\}$.
We make the following assumptions on  the polynomial $f$
(cf. \cite[Theorem 4.2]{PutVasA},
\cite[Definition 3]{Pow},
\cite[Definition 3]{EscPer}).

\begin{assumption}\label{assump:f} {\ }
\begin{enumerate}
\item $m=\deg_Y(f)$ is even.
\item $f_m(\bar{x})>0$ on $S_\infty$.
\end{enumerate}
\end{assumption}
Under Assumptions \ref{assump:g_i} and \ref{assump:f}, we have that $\wt{f} >0$ on $\wt{S} \cap \{ z=0\}$.

We are ready now to state our main result, using the notation we introduced above.

\begin{theorem}\label{thm:main}
Let $\bg:=g_1,\dots, g_s$ and $f$ be polynomials in $\R[\bar{X},Y]$
such that $f>0$ on $S$ and Assumptions
\ref{assump:arq}, \ref{assump:g_i} and \ref{assump:f} hold.
Let $q \in \R[Y]$ be a non-constant polynomial which is positive on $\R$.
Then, there exists $M\in \Z_{\ge 0}$ such that $q^M f \in M(\bg)$.
\end{theorem}

Note that, since $q \in \R[Y]$ is a sum of squares, multiplying on both sides by $q$ if necessary,
we may assume that $M$ is even. If
$q^Mf =  \sigma_0 + \sigma_1 g_1 + \cdots +\sigma_s g_s$
with $\sigma_0,\sigma_1, \dots, \sigma_s \in \sum \R[\bar{X}, Y]^2$,
then the identity
$$
f = \frac{\sigma_0 + \sigma_1 g_1 + \cdots +\sigma_s g_s}{q^M}
$$
is a rational certificate of non-negativity for $f$ on $S$.
For each $y \in \R$, this identity can be evaluated to
express $f(\bar X, y)$ as en explicit element of the quadratic module generated
by $g_1(\bar X, y), \dots, g_s(\bar X, y)$ in $\R[\bar X]$, thus obtaining a certificate of
non negativity on the slices of $S$ cut by the equations $Y = y$ in a parametric way.

The following example (\cite[Example 8]{EscPer}) shows that the second condition in Assumption \ref{assump:f} is necessary for the result to hold.

\begin{example} For $n = 1$ consider $g_1 = (1 - X^2)^3 \in \R[X, Y]$. Then $S = [-1, 1] \times \R \subset \R^2$ and
$$
\frac43 - X^2 = \frac43X^2\Big(X^2 - \frac32\Big)^2 + \frac43\Big(1-X^2\Big)^3 \in M(g_1)
$$
(see also \cite[Theorem 7.1.2]{MarBook}).
Take $f(X, Y) = (1 - X^2)Y^2 + 1 \in \R[X, Y]$. It is clear that $f > 0$,
however it is not the case that $f_2 = 1 - X^2$ is positive on $S_{\infty} = [-1, 1]$.

For any $q(Y) \in \R[Y]$ which is positive on $\R$, if we have an identity
\begin{equation*}
q(Y)^M((1-X^2)Y^2 + 1) = \sum_{j}\Big( \sum_{i}p_{ji}(X)Y^i \Big)^2 +
\sum_{j}\Big( \sum_{i}q_{ji}(X)Y^i \Big)^2
\big(1 - X^2\big)^3,
\end{equation*}
every term on the right hand side has degree in $Y$ bounded by
$2m'=m_0M+2$ (where $m_0 = \deg (q)$), and then looking at the terms of degree $2m'$ in $Y$, we have
$$
q_{m_0}^M(1-X^2) = \sum_{j }p_{jm'}(X)^2  +
\sum_{j } q_{jm'}(X)^2\big(1 - X^2\big)^3.
$$
This implies that $1-X^2$ belongs to the quadratic module generated by $(1-X^2)^3$ in $\R[X]$, which is false since it is the well-known example from \cite[Example]{Ste2}.
\end{example}

\section{Proof of the main result}\label{sec:proof}

As mentioned in the Introduction, the main idea to prove Theorem \ref{thm:main} is to produce in a parametric way, for each $y \in \R$ a certificate on the slice of $S$ cut by the equation $Y = y$. To this end,
we adapt the techniques in \cite{Schw} and \cite{NieSchw2007}
using \cite{PowersReznick2001}.
We keep the notation from the previous section.

Let $\Delta\subset \R^n$ be the following simplex containing $\bfB$:
\begin{equation}\label{eq:simplex}
\Delta = \{\bar{x}\in \R^n \mid x_j \ge -\sqrt{N} \hbox{ for } j=1,\dots, n, \sum_{1\le j\le n} x_j \le \sqrt{nN}\}.
\end{equation}

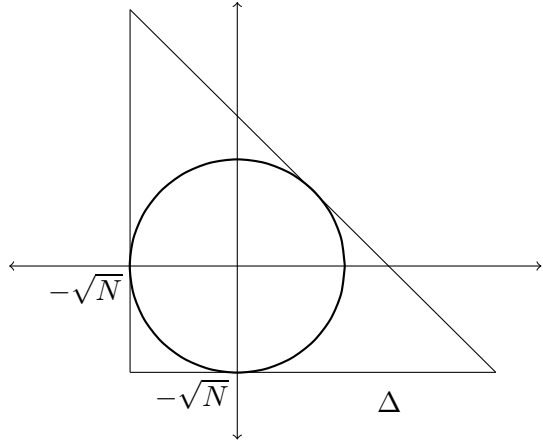
\begin{center}
\begin{tikzpicture}
      \draw[<->] (-3,0) -- (4,0);
      \draw[<->] (0,-2.3) -- (0,3.5);
      \draw[line width=0.8pt, domain= 0:360,smooth,variable=\x,]
       plot ( { sqrt(2)*cos(\x)  }, { sqrt(2)*sin(\x) }  );
      \draw (-1.41,-1.41) -- (-1.41,3.4);
      \draw (-1.41,-1.41) -- (3.4, -1.41);
      \draw (-1.41,3.4) -- (3.4, -1.41);
      \node at (2,-1.8) {$\Delta$};
      \node at (-2,-0.3) {$-\sqrt{N}$};
      \node at (-0.6,-1.7) {$-\sqrt{N}$};
       \end{tikzpicture}
\end{center}

Since $\bfC$ is compact and $0 \not \in \bfC$, there exist positive $\rho_1, \rho_2 \in \R$ which are respectively the minimum and maximum value of $\|(y, z)\|$ for $(y, z) \in \bfC$. Taking this into account, the following lemma
can be proved similarly as \cite[Lemma 11]{NieSchw2007}.

\begin{lemma}\label{rem:constLips}
Let $f \in \R[\bar{X}, Y]$. There is a constant
$K>0$
such that, for every $\xi_1, \xi_2 \in \Delta \times \bfC$,
\[|\wt{f} (\xi_1) - \wt{f}(\xi_2)| \le
K\, \|\xi_1-\xi_2 \|.\]
Moreover, the constant $K$ can be computed in terms
of $n$, the degrees in $\bar{X}$ and $Y$ of $f$, the size of the coefficients of $f$, $N$ and $\rho_2$.
\end{lemma}

As explained in the previous section, if $f$ is positive on $S$ and Assumptions \ref{assump:arq}, \ref{assump:g_i} and \ref{assump:f} are satisfied, then $\widetilde f$ is positive on
$$
\widetilde S \ \subset \
\mathbf{B} \times \bfC \ \subset \ \Delta \times \bfC.
$$
We denote
\[ f^\bullet:= \min\{ \wt{f}(\bar{x},y,z) \mid (\bar{x},y,z) \in  \wt{S}\} > 0.\]

Our first aim is to construct an auxiliary polynomial $h\in \R[\bar{X},Y,Z]$ such that
\begin{itemize}
\item $h(\bar{x},y,z)$ is positive on $\Delta \times \bfC$, 
\item $h(\bar{x},y,1) = q(y)^M f(\bar{x}, y) - p(\bar{x},y)$,  for a polynomial $p\in M(\mathbf{g})$.
\end{itemize}

Let $r$ be the remainder of $m = \deg_Y (f)$ in the division by $m_0 = \deg (q) > 0$ and,
for $i = 1, \dots, s$, let $e_i \in \Z$ be the minimum non-negative number such that
$m_0 $ divides $ m_i + e_i$.
We have then that $r, e_1, \dots, e_s$ are even and $0 \le r, e_1, \dots, e_s \le m_0-2$.

\begin{proposition} \label{prop:polyh} With our previous notation and assumptions, there exist $\lambda,
\alpha_1, \dots, \alpha_s \in \R_{>0}$, $k\in \Z_{\ge 0}$
and $M,M_1,\dots,M_s\in \Z_{\ge 0}$
such that the polynomial
$$
h(\bar{X},Y,Z) \  = \
\widetilde q(Y,Z)^{M}\, \widetilde{f}(\bar{X},Y,Z) \ -
$$
$$
- \lambda (Y^2 + Z^2)^{\frac{r}2} \displaystyle{\sum_{1\le i\le s}} \alpha_i(Y^2 + Z^2)^{\frac{e_i}2}\widetilde{g}_i(\bar{X},Y,Z)
(\alpha_i(Y^2 + Z^2)^{\frac{e_i}2}\widetilde{g}_i(\bar{X},Y,Z)-\widetilde q(Y,Z)^{\frac{m_i + e_i}{m_0}})^{2k}\widetilde q(Y,Z)^{M_{i}},
$$
is homogeneous in $(Y, Z)$ of degree
$\max\{m; r + (2k+1)(m_i+e_i), i=1,\dots, s\}$,
and satisfies
\[ h(\bar{x},y,z) > \frac{f^\bullet}{2} \]
for all  $(\bar{x},y,z)\in \Delta \times \bfC$.
\end{proposition}

Note that for $i = 1, \dots, s$, the degree in $(Y, Z)$ of
$$
(Y^2 + Z^2)^{\frac{r + e_i}2}\widetilde{g}_i(\bar{X},Y,Z)
(\alpha_i(Y^2 + Z^2)^{\frac{e_i}2}\widetilde{g}_i(\bar{X},Y,Z)-\widetilde q(Y,Z)^{\frac{m_i + e_i}{m_0}})^{2k}
$$
has remainder $r$ in the division by $m_0$. This ensures that, once $k$ is fixed, there exist unique $M, M_1, \dots, M_s$ so that the degree and homogeneity conditions on $(Y,Z)$ are satisfied.

\begin{proof}{Proof of Proposition \ref{prop:polyh}:}
Since $\Delta \times \bfC$ is compact, for $i=1,\dots, s$, we define
$$
\beta_i =
\sup_{\Delta \times \bfC} \ \Big|  (y^2 + z^2)^{\frac{e_i}2} \wt{g}_i(\bar x, y, z) \Big|,  \qquad \alpha_i =
\frac{1}{\beta_i + 1}
$$
and
$$
G_i(\bar X, Y, Z) = \alpha_i(Y^2 + Z^2)^{\frac{e_i}2} \wt{g}_i(\bar X, Y, Z).
$$
Then, for $i=1, \dots, s$ we have
$$
\sup_{\Delta \times \bfC} \ \Big|  G_i(\bar x, y, z) \Big| < 1.
$$
The polynomial $h$ in the statement of the proposition can be rewritten as
$$
h(\bar{X},Y,Z) \ = \
\widetilde q(Y,Z)^{M}\, \widetilde{f}(\bar{X},Y,Z) \ -
$$
$$
- \, \lambda (Y^2 + Z^2)^{\frac{r}2} \displaystyle{\sum_{1\le i\le s}}
G_i(\bar{X},Y,Z)
(G_i(\bar{X},Y,Z)-\widetilde q(Y,Z)^{\frac{m_i + e_i}{m_0}})^{2k}\widetilde q(Y,Z)^{M_{i}}.
$$

For every $(\bar{x},y,z) \in \Delta \times \bfC$, we have that $\widetilde q(y, z) =1$; then,
$$
h(\bar{x},y,z) = \wt{f}(\bar{x},y,z) - \lambda (y^2 + z^2)^{\frac{r}2} \sum_{1\le i \le s}  G_i(\bar{x},y,z) \left(G_i(\bar{x},y,z)-1\right)^{2k}.
$$
For $i = 1, \dots, s$, if $G_i(\bar{x},y,z)\ge 0$, then $0\le G_i(\bar{x},y,z)<1$. Now, it is not difficult to see that, for every $t\in [0,1]$, the inequality $t(t-1)^{2k} < \dfrac{1}{2ek}$ holds; therefore,
\begin{equation}\label{eq:bound_gi}
G_i(\bar{x},y,z) (G_i(\bar{x},y,z)-1)^{2k} < \frac{1}{2ek}.
\end{equation}

Consider the set
\[A= \left\{ (\bar{x},y,z)\in \Delta \times \bfC \ \mid \ \wt{f}(\bar{x},y,z) \le \frac{3}{4} f^\bullet \right\}.\]
Note that $A\cap \wt{S}= \emptyset$, since $\wt{f}(\bar{x},y,z) \ge f^\bullet$ for every $(\bar{x},y,z) \in \wt{S}$.

For $(\bar{x},y,z) \in (\Delta \times \bfC)-A$,
$$
h(\bar{x},y,z) \ge \wt{f}(\bar{x},y,z) - \lambda \rho_2^r\sum_{{1\le i \le s,}\atop{G_i(\bar{x},y,z) \ge 0} }
G_i(\bar{x},y,z) \left(G_i(\bar{x},y,z)-1\right)^{2k}
$$
and, as a consequence,
\[h(\bar{x},y,z) > \dfrac{3}{4} f^{\bullet} -  \dfrac{\lambda \rho_2^rs}{2ek}.\]
Therefore, if we take $k\ge \dfrac{2\lambda \rho_2^rs}{e f^\bullet}$, we have
\[h(\bar{x},y,z) > \dfrac{ f^\bullet}{2}.\]

For $\bar{\xi} =( \bar{x},y,z ) \in \Delta\times \bfC$, we define $H(\bar{\xi}) = \mbox{dist}(\bar{\xi}, \wt{S})$ and $F(\bar{\xi})= - \min \{0, G_1(\bar{\xi}), \dots, G_s(\bar{\xi})\}$. Note that
$$
F^{-1}(0)=\{\bar{\xi} \in \Delta \times \bfC \ \mid \
G_1(\bar{\xi}) \ge 0, \dots, G_s(\bar{\xi})\ge 0 \}
=
$$
$$
=
\{\bar{\xi} \in \Delta \times \bfC \ \mid \ \wt{g}_1(\bar{\xi}) \ge 0, \dots, \wt{g}_s(\bar{\xi})\ge 0 \}
=
\wt{S} = H^{-1}(0).
$$
By the {\L}ojasiewicz inequality (see \cite[Corollary 2.6.7]{BCR1997}), there exist positive constants $L$ and $c$ such that, for all $\bar{\xi} \in \Delta\times \bfC$,
\[ \mbox{dist}(\bar{\xi}, \wt{S})^L \le c \, F(\bar{\xi}).\]

Since $A \cap \wt{S} = \emptyset$, for $\bar{\xi} \in A$,  $F(\bar{\xi}) > 0$.
Let $i_0$, with $1\le i_0\le s$, be such that $F(\bar{\xi})=-G_{i_0}(\bar{\xi})$;
then,
\[G_{i_0}(\bar{\xi}) \le -\dfrac{1}{c} \, \mbox{dist}(\bar{\xi}, \wt{S})^L.\]
Let $\bar{\xi}_0\in \wt{S}$ be a point where the distance from $\bar{\xi}$ to $\wt{S}$ is attained, namely,  $\mbox{dist}(\bar{\xi}, \wt{S}) = \|\bar{\xi} - \bar{\xi}_0\|$. As
$\frac{f^\bullet}{4} \le \wt{f}(\bar{\xi}_0) - \wt{f}(\bar{\xi}) \le
K\|\bar{\xi}_0 - \bar{\xi} \|,$ where
$K$ is the positive constant from Lemma \ref{rem:constLips},
we deduce that
\[ \mbox{dist}(\bar{\xi}, \wt{S})= \|\bar{\xi}_0 - \bar{\xi} \| \ge \dfrac{f^\bullet}{4K}
\]
and, as a consequence,
\[G_{i_0}(\bar{\xi}) \le - \dfrac{1}{c} \left(\dfrac{f^\bullet}{4K}
\right)^L. \]
Together with inequality \eqref{eq:bound_gi}, this implies that
\begin{align*}
h(\bar{\xi})
&\ge \wt{f}(\bar{\xi})+ \dfrac{\lambda \rho_1^r}{c}\left(\dfrac{f^\bullet}{4K}
 \right)^L- \dfrac{\lambda \rho_2^r(s-1)}{2ek} \\
 {} & = \left(\wt{f}(\bar{\xi}) - f^\bullet + \dfrac{\lambda \rho_1^r}{c}\left(\dfrac{f^\bullet}{4K}
 \right)^L \right)+ \left(f^\bullet- \dfrac{\lambda \rho_2^r (s-1)}{2ek}\right)
\end{align*}

Let $\bar{\xi}^\bullet\in \wt{S}$ be a point where the minimum $f^\bullet$ of $\wt{f}$ is attained in $\wt{S}$, that is, $\wt{f}(\bar{\xi}^\bullet) = f^\bullet$. By Lemma \ref{rem:constLips}, we have
\[|\wt{f}(\bar{\xi}) - f^\bullet| =| \wt{f}(\bar{\xi}) -\wt{f}(\bar{\xi}^\bullet)| \le
K\| \bar{\xi} - \bar{\xi}^\bullet\| \]
and so, if $D:= \mbox{diam}(\Delta\times \bfC)$,
\[|\wt{f}(\bar{\xi}) - f^\bullet| \le
K  D.\]
Therefore, for $\lambda \ge K D  \dfrac{c}{\rho_1^r} \left(
\dfrac{4 K}{f^\bullet}\right)^L$,
we have that
\begin{equation}\label{eq:bound1}
\wt{f}(\bar{\xi}) - f^\bullet + \dfrac{\lambda \rho_1^r}{c}\left(\dfrac{f^\bullet}{4K}
\right)^L \ge 0.
\end{equation}
On the other hand, for $k\ge \dfrac{2 \lambda \rho_2^r s}{ef^\bullet}$, the inequalities
\[\dfrac{\lambda \rho_2^r(s-1)}{2ek}\le \dfrac{f^\bullet}{4}\, \dfrac{(s-1)}{s}  < \dfrac{f^\bullet}{4}\]
hold and so,
\begin{equation}\label{eq:bound2}
f^\bullet- \dfrac{\lambda \rho_2^r(s-1)}{2ek} > \dfrac{3}{4} f^\bullet.
\end{equation}
From \eqref{eq:bound1} and \eqref{eq:bound2}, we conclude that \[h(\bar{\xi}) >\dfrac{3}{4} f^\bullet.\]

Summarizing, for $\lambda \ge K D  \dfrac{c}{\rho_1^r} \left(
\dfrac{4 K}{f^\bullet}\right)^L$
and $k\ge \dfrac{2 \lambda \rho_2^rs}{ef^\bullet}$, we have that $h(\bar{x},y,z) > \dfrac{f^\bullet}{2}$ for every $(\bar{x},y,z) \in \Delta \times \bfC$.
\end{proof}

\begin{remark}\label{rem:degbound} From the proof of Proposition \ref{prop:polyh}, it follows that, once the polynomials $g_1, \dots, g_s$ and $q$ are fixed,
for every $f$ positive on $S$ satisfying Assumptions \ref{assump:arq}, \ref{assump:g_i}
and \ref{assump:f},
an explicit bound for $M$
in terms of $\deg (f)$, the size
of the coefficients of $f$ and $f^{\bullet}$
can be computed similarly as in \cite{NieSchw2007} or
\cite{EscPer}.
\end{remark}

In order to prove our main result, we will apply the following effective version of Polya's theorem for a simplex (see \cite[Theorem 3]{PowersReznick2001}).

\begin{lemma} \label{lem:Polya} Let $P\subset \R^n$ be an $n$-dimensional simplex with vertices $v_0,\dots, v_n$, and let $\bar{\ell}=\{\ell_0,\dots, \ell_n\}$ be the set of barycentric coordinates on $P$, i.e., $\ell_i \in \R[\bar{X}]$ is linear (affine) for $i=0,\dots, n$,
\[\bar{X} = \sum_{0\le i\le n} \ell_i(\bar{X}) v_i , \quad 1 =\sum_{0\le i\le n} \ell_i(\bar{X}),  \quad \hbox{and} \quad \ell_i(v_j) = \delta_{ij} \ \hbox{for } 0\le i,j\le n.\]
Let $h\in \R[\bar{X}]$ be a polynomial  that is strictly positive on $P$.  Then, for $\kappa \gg 0$, $h$ has a representation of the form
\[h= \sum_{|\beta| \le \kappa } b_\beta\, \bar{\ell}^\beta \qquad \hbox{with} \ \ b_\beta > 0.\]
Moreover, for each $\beta$,  $b_\beta\in \R$ is a linear combination of the coefficients of $h$, and an explicit bound for
$\kappa$ can be given in terms of the degree of $h$, the size of the coefficients of $h$, the minimum value of $h$ in $P$
and the vertices $v_0,\dots, v_n$.
\end{lemma}

\begin{lemma}\label{lem:mcball} For $N \in \R_{>0}$,
let $\ell_0(\bar{X}):= \sqrt{nN} - \sum_{1\le j\le n} X_j$ and, for $i=1,\dots, n$, $\ell_i(\bar{X}) := X_i +\sqrt{N}$. Then,  for $i=0,\dots, n$, we have that $\ell_i \in M(N - \| \bar{X}\|^2)$, where $\| \bar{X}\|^2 = \sum_{1\le j \le n}X_j^2$.
\end{lemma}

\begin{proof}{Proof:}
By an explicit computation, we see that
\[\sqrt{nN} -\sum_{1\le j\le n} X_j = \frac{1}{2\sqrt{nN}}\Big((\sqrt{nN}  - \sum_{1\le j\le n} X_j)^2 +
\sum_{1 \le j < j' \le n}(X_j - X_{j'})^2\Big)+\frac{\sqrt{n}}{2\sqrt{N}} (N - \sum_{1\le j \le n} X_j^2).\]
Also,  for $i=1,\dots, n$:
\[X_i + \sqrt{N} = \frac{1}{2\sqrt{N}} \Big((X_i+ \sqrt{N})^2+\sum_{j\ne i}X_j^2\Big) +  \frac{1}{2\sqrt{N}} \Big(N -\sum_{1\le j \le n}X_j^2\Big).\]
This shows that $\ell_0,\ell_1, \dots, \ell_n \in M(N - \| \bar{X}\|^2)$.
\end{proof}

We are now able to prove the main result of the paper.

\begin{proof}{Proof of Theorem \ref{thm:main}:}
We continue to use  the notation introduced before.

Let $h\in \R[\bar{X},Y,Z]$ be as in Proposition \ref{prop:polyh}.
We will apply P\'olya's theorem, as stated in Lemma \ref{lem:Polya}, to the polynomials
$h_{(y,z)}(\bar{X}):=h(\bar{X},y,z)$
for $(y,z)\in \bfC$ and the simplex $\Delta$
defined in (\ref{eq:simplex}).

The vertices of $\Delta$ are
\begin{align*}
v_0 &:=(-\sqrt{N}, \dots,-\sqrt{N}), \\
v_i &:=
v_0 + (0, \dots, \underbrace{(n+\sqrt{n})\sqrt{N}}_{i-{\rm th\ coord.}}, \dots, 0)
\quad \hbox{ for } i=1, \dots, n,
 \end{align*}
and its barycentric coordinates are given by
\begin{align*}
 \ell_0(\bar{X})&:= \frac{1}{(n+\sqrt{n})\sqrt{N}}(\sqrt{nN}-\sum_{1\le j\le n} X_j),\\
 \ell_i(\bar{X})&:=\frac{1}{(n+\sqrt{n})\sqrt{N}} (X_i+\sqrt{N}) \quad \hbox{ for } i=1,\dots, n.
 \end{align*}

For each fixed $(y,z) \in \bfC$,
the polynomial $h_{(y,z)}(\bar{X})$ satisfies
\[h_{(y,z)}(\bar{x}) >\dfrac{f^\bullet}{2}> 0 \quad \hbox{for every} \ \bar{x} \in \Delta.\]
Since the size of the coefficients of $h_{(y,z)}$ as
polynomials in $\bar X$ and their minimum values on $\Delta$
are uniformly bounded for $(y, z) \in \bfC$,
by
Lemma \ref{lem:Polya}, there exists $\kappa\gg 0$ such that
\[ h(\bar{X}, Y, Z) = \sum_{|\beta|\le \kappa} b_\beta(Y,Z) \bar{\ell}(\bar{X})^{\beta}\]
with
$b_\beta \in \R[Y, Z]$ and
$b_\beta(y,z) >0$ for every $(y,z) \in \bfC$. In addition, for each $\beta$,
since $h$ is homogeneous in $(Y, Z)$ and $b_\beta$
is a linear combination of the coefficients of $h$ (seen as a polynomial in $\bar X$), then $b_\beta$ is
a homogeneous polynomial. This implies that
$b_\beta(y,z) >0$ for every $(y,z) \in \R \times \R_{\ge 0} \setminus \{0\}$; in particular, $b_\beta(y,1) >0$ for every $y\in \R$ and therefore $b_\beta(Y, 1)$ is a sum of squares in
$\R[Y]$.

Finally, from the equality
$$
h(\bar{X},Y,1) =q(Y)^{M}\, f(\bar{X},Y) \ -
$$
$$
-\lambda \sum_{1\le i\le s} \alpha_i (Y^2 + 1)^{\frac{r + e_i}2}{g}_i (\bar{X},Y)
(\alpha_i(Y^2 + 1)^{\frac{e_i}2}{g}_i(\bar{X},Y)-q(Y)^{\frac{m_i+e_i}{m_0}})^{2k} q(Y)^{M_{i}}
$$
we have:
$$
q(Y)^{M}\, f(\bar{X},Y) = \lambda
\sum_{1\le i\le s} \alpha_i (Y^2 + 1)^{\frac{r + e_i}2}
(\alpha_i(Y^2 + 1)^{\frac{e_i}2}{g}_i(\bar{X},Y)-q(Y)^{\frac{m_i+e_i}{m_0}})^{2k} q(Y)^{M_{i}}{g}_i (\bar{X},Y)
$$
$$
+ \sum_{|\beta|\le \kappa} b_\beta(Y,1) \bar{\ell}(\bar{X})^{\beta}.
$$

It is clear that for $i = 1, \dots, s$,
$$
\alpha_i (Y^2 + 1)^{\frac{r + e_i}2}
(\alpha_i(Y^2 + 1)^{\frac{e_i}2}{g}_i(\bar{X},Y)-q(Y)^{\frac{m_i+e_i}{m_0}})^{2k} q(Y)^{M_{i}}{g}_i (\bar{X},Y) \in M(\bg).
$$
On the other hand, by Lemma \ref{lem:mcball},
\[\ell_0(\bar{X}),\dots, \ell_n(\bar{X}) \in M(N-\|\bar{X}\|^2)\] and, taking into account that $M(N-\|\bar{X}\|^2)$ is closed under multiplication (since it is generated by a single polynomial), the same holds for all the products $\bar{\ell}(\bar{X})^\beta = \ell_0(\bar{X})^{\beta_0}\cdots \ell_n(\bar{X})^{\beta_n}$.
By the assumption $N-\|\bar{X}\|^2\in M(\bg)$, we deduce that $b_\beta(Y,1) \bar{\ell}(\bar{X})^{\beta} \in M(\bg)$ for every $\beta$ with $|\beta|\le \kappa$. We conclude that $q(Y)^{M}\, f(\bar{X},Y)\in M(\bg)$.
\end{proof}

\begin{remark}
The value of $M$ in Theorem \ref{thm:main} is the same as in Proposition \ref{prop:polyh},
therefore it can be bounded as mentioned in Remark \ref{rem:degbound}.
\end{remark}


\begin{thebibliography}{00}


\bibitem{Artin} E. Artin, Uber die Zerlegung definiter Funktionen
in Quadrate. \emph{Abh. Math. Sem. Hamburg} 5 (1927), no. 1, 100--115.

\bibitem{BCR1997}
J. Bochnak, M. Coste, M.-F. Roy,  Real algebraic geometry. Ergebnisse der Mathematik und ihrer
Grenzgebiete 36, Springer, Berlin, 1998.



\bibitem{EscPer} P. Escorcielo, D. Perrucci, A version of Putinar's Positivstellensatz for cylinders. \emph{J. Pure Appl. Algebra} 224 (2020), no. 12, 106448.


\bibitem{Kri} J.-L. Krivine,  Anneaux pr\'eordonn\'es. {\em Journal
d'analyse math\'ematique} 12 (1964),  307--326.


\bibitem{KuhlMar}
S. Kuhlmann, M. Marshall, Positivity, sums of squares and the multi-dimensional moment problem.
\emph{Trans. Amer. Math. Soc.} 354 (2002), no. 11, 4285--4301.

\bibitem{KuhlMarSchw2005}
S. Kuhlmann, M. Marshall, N. Schwartz,
Positivity, sums of squares and the multi-dimensional moment problem II.
\emph{Adv. Geom.} 5 (2005), no.4,  583--606.



\bibitem{MarBook}
M. Marshall,
Positive polynomials and sums of squares.
Mathematical Surveys and Monographs, 146. American Mathematical Society, Providence, RI, 2008.



\bibitem{NieSchw2007} J. Nie, M. Schweighofer,
On the complexity of Putinar's Positivstellensatz.
\emph{J. Complexity} 23 (2007), no. 1, 135--150.

\bibitem{Pow} V. Powers,
Positive polynomials and the moment problem for cylinders with compact cross-section.
\emph{J. Pure Appl. Algebra} 188 (2004), no. 1-3, 217--226.


\bibitem{PowersBook} V. Powers,
Certificates of positivity for real polynomials - theory, practice, and applications. Developments in Mathematics, 69. Springer, Cham, 2021.

\bibitem{PowersReznick2001}
V. Powers, B. Reznick, A new bound for P\'olya's Theorem with applications to polynomials positive on polyhedra. \emph{J. Pure Appl. Algebra} 164 (2001), no. 1-2, 221--229.


\bibitem{Put}  M. Putinar,
Positive polynomials on compact semi-algebraic sets.
\emph{Indiana Univ. Math. J.} 42 (1993), no. 3, 969--984.



\bibitem{PutVasA} M. Putinar, F.-H. Vasilescu, Solving moment problems by dimensional extension.
\emph{Ann. of Math. (2)} 149 (1999), no. 3, 1087--1107.


\bibitem{PutVasB} M. Putinar, F.-H. Vasilescu, Positive polynomials on semi-algebraic sets.
\emph{C. R. Acad. Sci. Paris Sér. I Math.} 328 (1999), no. 7, 585--589.



\bibitem{Reznick95} B. Reznick,  Uniform denominators in Hilbert's seventeenth problem.
\emph{Math. Z.} 220 (1995), no. 1, 75--97.



\bibitem{ScheidererSurvey} C. Sheiderer,
 Positivity and sums of squares: a guide to recent results.
 \emph{Emerging applications of algebraic geometry}, 271–324, IMA Vol. Math. Appl., 149, Springer, New York, 2009.

\bibitem{Schm} K. Schm\"udgen,
The K-moment problem for compact semi-algebraic sets.
\emph{Math. Ann.} 289 (1991), no. 2, 203--206.
 

  
 
 
 \bibitem{Schm2003} K. Schm\"udgen, On the moment problem of closed semi-algebraic sets. 
 \emph{J. Reine Angew. Math.} 558 (2003), 225--234. 


\bibitem{Schm2017} K. Schm\"udgen, A general fibre theorem for moment problems and some applications. \emph{Israel J. Math.} 218 (2017), no. 1, 43--66.
 
 
 \bibitem{SchmBook} K. Schm\"udgen, The moment problem. Graduate Texts in Mathematics, 277. 
\emph{Springer, Cham}, 2017.
 
 


\bibitem{Schw} M. Schweighofer,
On the complexity of Schm\"udgen's Positivstellensatz.
\emph{J. Complexity} 20 (2004), no. 4, 529--543.



\bibitem{Ste} G. Stengle, A Nullstellensatz and a
Positivstellensatz in semialgebraic geometry.
\emph{Math. Ann.} 207 (1974), 87--97.

\bibitem{Ste2}
G. Stengle,
Complexity estimates for the Schm\"udgen Positivstellensatz.
\emph{J. Complexity} 12 (1996), no. 2, 167--174.



\end{thebibliography}
\end{document}